\documentclass[a4paper,10pt]{amsart}

\usepackage{graphicx}

\usepackage{hyperref}

\hypersetup{%
pdftitle={},
pdfsubject={Mathematics},
pdfauthor={Manfred Madritsch},
pdfkeywords={}
hyperindex=true,plainpages=false}

\usepackage{a4wide}

\usepackage{amsmath}
\usepackage{amsfonts}
\usepackage{amssymb}
\usepackage{amsthm}

\newtheorem{lem}{Lemma}[section]
\newtheorem{thm}[lem]{Theorem}
\newtheorem{prop}[lem]{Proposition}

\numberwithin{equation}{section}

\newtheorem*{cor*}{Corollary}
\newtheorem*{thm*}{Theorem}

\theoremstyle{definition}
\newtheorem{defi}{Definition}[section]

\theoremstyle{remark}

\DeclareMathOperator{\Tr}{Tr}

\allowdisplaybreaks[3]

\newcommand{\N}{\mathbb{N}}
\newcommand{\Z}{\mathbb{Z}}
\newcommand{\Q}{\mathbb{Q}}
\newcommand{\R}{\mathbb{R}}

\newcommand{\lf}{\left\lfloor}
\newcommand{\rf}{\right\rfloor}
\renewcommand{\lvert}{\left\vert}
\renewcommand{\rvert}{\right\vert}
\renewcommand{\lVert}{\left\Vert}
\renewcommand{\rVert}{\right\Vert}
\newcommand{\house}[1]{\mathrm{H}\left(#1\right)}

\title[Asymptotic normality in
canonical number systems]{Asymptotic normality of additive functions on polynomial
  sequences in canonical number systems}

\author[M. G. Madritsch]{Manfred G. Madritsch}
\address[M. G. Madritsch]{
Department of Analysis and Computational Number Theory\newline
\indent Graz University of Technology\newline
\indent A-8010 Graz, Austria}
\email{madritsch@math.tugraz.at}

\author[A. Peth{\H o}]{Attila Peth{\H o}}
\address[A. Peth{\H o}]{Department of Computer Science, University of Debrecen,\newline
\indent Number Theory Research Group,\newline
\indent Hungarian Academy of Sciences and University of Debrecen
\newline
\indent P.O. Box 12, H-4010 Debrecen, Hungary}
\email{petho.attila@inf.unideb.hu}

\date{\today}

\begin{document}

\begin{abstract}
The objective of this paper is the study of functions which only act on the
digits of an expansion. In particular, we are interested in the asymptotic
distribution of the values of these functions. The presented result is an
extension and generalization of a result of Bassily and K\'atai to number
systems defined in a quotient ring of the ring of polynomials over the
integers.
\end{abstract}

\maketitle

\section{Introduction}\label{sec:introduction}
In this paper we investigate the asymptotic behaviour of $q$-additive
functions. But before we start we need an idea of additive functions
and the number systems they are living in. Note that a function $f$ is
said to be \textit{$q$-additive} if it acts only on the $q$-adic digits,
\textit{i.e.}, $f(0)=0$ and
\[
  f(n)=\sum_{h=0}^\ell f(a_h(n)q^h)\quad\text{for}\quad n=\sum_{h=0}^\ell a_h(n)q^h,
\]
where $a_h(n)\in\mathcal{N}:=\{0,\ldots,q-1\}$ are the \textit{digits} of the
$q$-adic expansion of $n$.

One of the first results dealing with the asymptotic behavior of such a
$q$-additive function is the following, which is due to Bassily and
K\'atai~\cite{Bassily_Katai1995:distribution_values_q}.
\begin{thm*}
Let $f$ be a $q$-additive function such that $f(aq^h)=\mathcal{O}(1)$ as
$h\to\infty$ and $a\in\mathcal{N}$. Furthermore let
\[
  m_{h,q}:=\frac1q\sum_{a\in\mathcal{N}}f(aq^h),\quad
  \sigma^2_{h,q}:=\frac1q\sum_{a\in\mathcal{N}}f^2(aq^h)-m^2_{h,q},
\]
and
\[
  M_q(x):=\sum_{h=0}^Nm_{h,q},\quad
  D^2_q(x)=\sum_{h=0}^N\sigma^2_{h,q}
\]
with $N=[\log_qx]$. Assume that $D_q(x)/(\log x)^{1/3}\to\infty$ as
$x\to\infty$ and let $P$ be a polynomial with integer coefficients, degree
$d$ and positive leading term. Then, as $x\to\infty$,
\[
\frac1x\#\left\{n<x\middle\vert
  \frac{f(P(n))-M_q(x^d)}{D_q(x^d)}<y\right\}\to\frac1{\sqrt{2\pi}}\int_{-\infty}^y\exp(-x^2)\mathrm{d}x.
\]
\end{thm*}

A first step towards a generalization of this concept is based on number
systems living in an order in an algebraic number field.
\begin{defi}
Let $\mathcal{R}$ be an integral domain, $b\in\mathcal{R}$, and
$\mathcal{N}=\{n_1,\ldots,n_m\}\subset\Z$. Then we call the pair
$(b,\mathcal{N})$ a \textit{number system} in $\mathcal{R}$ if every
$g\in\mathcal{R}$ admits a unique and finite representation of the
form
\begin{gather*}
  g=\sum_{h=0}^\ell a_h(g)b^h\quad\text{with}\quad a_h(g)\in\mathcal{N}
\end{gather*}
and $a_h(g)\neq0$ if $h\neq0$. We call $b$ the \textit{base} and
$\mathcal{N}$ the \textit{set of digits}.

If $\mathcal{N}=\mathcal{N}_0=\{0,1,\ldots,m\}$ for $m\geq1$ then we
call the pair $(b,\mathcal{N})$ a \textit{canonical number system}.
\end{defi}

When extending the number system to the complex plane one has to face
effects such as amenability, \textit{i.e.}, there may exist two or more
different expansions of one number. In fact, one can construct a graph (the
connection graph) which characterizes all the amenable expansions. This has
been done by M\"uller \textit{et al.}
\cite{mueller01:_fract_proper_of_number_system} (with a direct approach) and by
Scheicher and Thuswaldner \cite{mr1900260} (consideration of the odometer).

A different view on digits in number systems is done by normal numbers. These
are numbers in which expansion every possible block occurs asymptotically
equally often. Constructions of such numbers have been considered by Dumont
\textit{et al.} \cite{mr1737265} and the first author in
\cite{Madritsch2007:note_on_normal,Madritsch:generating_normal_numbers}

In this paper we mainly concentrate on additive functions. Thus we define
additive functions in these number systems as follows.
\begin{defi}
Let $(b,\mathcal{N})$ be a number system in the integral domain
$\mathcal{R}$. A function $f$ is called $b$-\textit{additive} if $f(0)=0$ and
\begin{gather*}
f(g)=\sum_{h\geq0}f(a_h(g)b^h)\quad\text{for}\quad
  g=\sum_{h=0}^\ell a_h(g)b^h.
\end{gather*}
\end{defi}

The simplest version of an additive function is the
sum-of-digits function $s_b$ defined by
\[
s_b(g):=\sum_{h\geq0}a_k(g).
\]

The result by Bassily and K\'atai was first generalized to number systems in
the Gaussian integers by Gittenberger and Thuswaldner \cite{mr1796518} who
gained the following

\begin{thm*}
Let $b\in\Z[i]$ and $(b,\mathcal{N})$ be a canonical number system in $\Z[i]$. Let $f$ be a
$b$-additive function such that $f(ab^h)=\mathcal{O}(1)$ as $h\to\infty$ and
$a\in\mathcal{N}$. Furthermore let
\[
  m_{h,b}:=\frac1{\mathrm{N}(b)}\sum_{a\in\mathcal{N}}f(ab^h),\quad
  \sigma^2_{h,b}:=\frac1{\mathrm{N}(b)}\sum_{a\in\mathcal{N}}f^2(ab^h)-m^2_{h,b},
\]
and
\[
  M_b(x):=\sum_{h=0}^Lm_{h,b},\quad
  D^2_b(x)=\sum_{h=0}^L\sigma^2_{h,b}
\]
with $\mathrm{N}$ the norm of an element over $\Q$ and
$L=[\log_{\mathrm{N}(b)}x]$.

Assume that $D_b(x)/(\log x)^{1/3}\to\infty$ as $x\to\infty$ and let $P$ be
a polynomial of degree $d$ with coefficients in $\Z[i]$. Then, as $N\to\infty$,
\[
\frac1{\#\left\{z\in\Z[i]\middle\vert\mathrm{N}(z)<N\right\}}
\#\left\{\mathrm{N}(z)<N\middle\vert\frac{f(P(z))-M_b(N^d)}{D_b(N^d)}<y
\right\}\to\frac1{\sqrt{2\pi}}\int_{-\infty}^y\exp(-x^2)\mathrm{d}x,
\]
where $\mathrm{N}$ is the norm of an element over $\Q$ and $z$ runs over the
Gaussian integers.
\end{thm*}

This build the base for further considerations of $b$-additive functions
in algebraic number fields in general. Therefore let
$\mathcal{K}=\Q(\beta)$ be an algebraic number field and denote by
$\mathcal{O}_{\mathcal{K}}$ its ring of integers (aka its maximal
order). Furthermore let $\beta\in\mathcal{O}_{\mathcal{K}}$ then we set
$\mathcal{R}=\Z[\beta]$ to be an order in $\mathcal{K}$. We now want to analyze
additive functions for number systems in $\mathcal{R}$.

We need some more parameters in order to successfully generalize the theorem
from above. Thus let $\mathcal{K}^{(\ell)}$ ($1\leq\ell\leq r_1$)
be the real conjugates of $\mathcal{K}$, while $\mathcal{K}^{(m)}$ and
$\mathcal{K}^{(m+r_2)}$ ($r_1< m\leq r_1+r_2$) are the pairs of complex
conjugates of $\mathcal{K}$, where $r_1+2r_2=n$.

For $\gamma\in\mathcal{K}$ we denote by $\gamma^{(i)}$ ($1\leq i\leq
n$) the conjugates of $\gamma$. In order to extend the term of
conjugation to the completion $\overline{\mathcal{K}}$ of
$\mathcal{K}$ we define for $\gamma_j\in\mathcal{K}$ and $x_j\in\R$
($1\leq j\leq n$) $\lambda=\sum_{1\leq j\leq n}x_j\gamma_j$ and
$\lambda^{(i)}:=\sum_{1\leq j\leq   n}x_j\gamma^{(i)}_j$.


Next we have to guarantee that we choose the increasing set for our asymptotic
distribution. In the integer case we had the logarithm of the value, since the
length of expansion growth with the logarithm. Since $\mathcal{R}$ is of
dimension $n$ we need a way to enlarge the area under consideration such that
the expansion growth also in a smooth way. This is motivated by the following
\begin{lem}[{\cite[Theorem]
  {Kovacs_Petho1992:representation_algebraic_integers}}]\label{kv:thm}
Let $\ell(\gamma)$ be the length of the expansion of $\gamma$ to the base
$b$. Then
\[
  \lvert\ell(\gamma)-\max_{1\leq i\leq n}\frac{\log\lvert\gamma^{(i)}\rvert}
  {\log\lvert b^{(i)}\rvert}\rvert\leq C.
\]
\end{lem}

Therefore we define $\mathcal{R}(T_1,\ldots,T_r)$ to be the set
\begin{gather}\label{mani:N}
  \mathcal{R}(T_1,\ldots,T_r):=\left\{\lambda\in\mathcal{R}:\lvert\lambda^{(i)}\rvert\leq
    T_i,1\leq i\leq r\right\}.
\end{gather}
Now we use Lemma \ref{kv:thm} to bound the area $\mathcal{R}(T_1,\ldots,T_n)$ such
that we reach all elements of a certain length. Thus for a fixed $T$ we set
$T_i$ for $1\leq i\leq n$ such that
\begin{gather}\label{mani:Ti}
  \log T_i=\log T\frac{\log\lvert b^{(i)}\rvert^n}
  {\log\lvert\mathrm{N}(b)\rvert}.
\end{gather}
Furthermore we will write for short $\mathcal{R}(\mathbf{T}):=\mathcal{R}(T_1,\ldots,T_r)$
with $T_i$ as in \eqref{mani:Ti}.

Finally one can extend the definition of a number system also for
negative powers of $b$. Then for $\gamma\in\overline{\mathcal{K}}$ such that
\[
  \gamma=\sum_{h=-\infty}^\ell a_hb^h\quad\text{with}\quad a_h\in\mathcal{N}
\]
we call
\[
  \lf\gamma\rf:=\sum_{h=0}^ha_hb^h
  \quad\text{and}\quad
  \{\gamma\}:=\sum_{h\geq1}a_hb^{-h}
\]
the \textit{integer part} and \textit{fractional part} of $\gamma$,
respectively.

With all these tools we now can state the generalization of the theorem of
Bassily and K\'atai to arbitrary number fields.

\begin{thm}[{\cite{madritsch2010:asymptotic_normality_b}}] \label{Madritsch}
Let $(b,\mathcal{N})$ be a number system in $\mathcal{R}$ and $f$ be a
$b$-additive function such that $f(ab^h)=\mathcal{O}(1)$ as
$h\to\infty$ and $a\in\mathcal{N}$. Furthermore
let
\[
  m_{h,b}:=\frac1{\mathrm{N}(b)}\sum_{a\in\mathcal{N}}f(ab^h),\quad
  \sigma^2_{h,b}:=\frac1{\mathrm{N}(b)}\sum_{a\in\mathcal{N}}f^2(ab^h)-m^2_{h,b},
\]
and
\[
  M_b(x):=\sum_{h=0}^Lm_{h,q},\quad
  D^2_b(x)=\sum_{h=0}^L\sigma^2_{h,q}
\]
with $L=[\log_{\mathrm{N}(b)}x]$.

Assume that there exists an $\varepsilon>0$ such that $D_b(x)/(\log
x)^{\varepsilon}\to\infty$ as $x\to\infty$ and let
$P\in\overline{\mathcal{K}}[X]$ be a polynomial of degree $d$. Then, as
$T\to\infty$ let $T_i$ be as in \eqref{mani:Ti},
\[
\frac1{\#\mathcal{R}(\mathbf{T})}
\#\left\{z\in \mathcal{R}(\mathbf{T})\middle\vert\frac{f(\lf
    P(z)\rf)-M_b(T^d)}{D_b(T^d)}<y \right\}
\to\frac1{\sqrt{2\pi}}\int_{-\infty}^y\exp(-x^2)\mathrm{d}x.
\]
\end{thm}

\section{Definitions and result}\label{sec:definitions-result}
The objective of this paper are generalizations of number systems to
quotient rings of the ring of polynomials over the integers. Our aim is to
extend Theorem \ref{Madritsch} to such rings. To formulate our results we have
to introduce the relevant notions. In particular we use the following
definition in order to describe number systems in quotient rings of the ring of
polynomials over the integers.

\begin{defi}
Let $p\in\Z[X]$ be monic of degree $n$ and let $\mathcal N$ be a subset of
$\Z$. The pair $(p,{\mathcal N}) $ is called a number system if for every
$g\in \Z[X]\setminus \{0\}$ there exist unique $\ell\in\N$ and $a_i\in
{\mathcal N}, h=0,\dots,\ell; a_{\ell}\not= 0$ such that
\begin{gather}\label{mani:Xrepresentation}
g \equiv \sum_{h=0}^{\ell} a_h(g) X^h \pmod{p}.
\end{gather}
In this case $a_i$ are called the digits and $\ell=\ell(a)$ the length of the
representation.
\end{defi}

This concept was introduced in
\cite{pethoe1991:polynomial_transformation_and} and was studied among
others in
\cite{akiyama_borbely_brunotte+2005:generalized_radix_representations,
akiyama_rao2004:new_criteria_canonical,
kovacs_petho1991:number_systems_in,
Kovacs_Petho1992:representation_algebraic_integers}.
It was proved in \cite{akiyama_rao2004:new_criteria_canonical}, that $\mathcal
N$ must be a complete residue system modulo $p(0)$ including $0$ and the zeroes
of $p$ are lying outside or on the unit circle. However, following the
argument of the proof of Theorem 6.1 of
\cite{pethoe1991:polynomial_transformation_and}, which dealt with the case
$p$ square free, one can prove that non of the zeroes of $p$ are lying on
the unit circle.

If $p$ is irreducible then we may replace $X$ by one of the roots
$\beta$ of $p$. Then we are in the case of $\Z[X]/(p)\cong\Z[\beta]$
being an integral domain in an algebraic number field (\textit{cf.} Section
\ref{sec:introduction}). Then we may also denote the number system by the
pair$(\beta,\mathcal{N})$ instead of $(p,\mathcal{N})$. For example, let
$q\geq2$ be a positive integer, then $(p,\mathcal{N})$ with $p=X-q$ gives a
number system in $\Z$, which corresponds to the number systems
$(q,\mathcal{N})$. Furthermore for $n$ a positive integer and
$p=X^2+2n\,X+(n^2+1)$ we get number systems in $\Z[i]$.

Now we want to come back to these more general number systems and consider
additive functions within them.
\begin{defi}
Let $(p,\mathcal{N})$ be a number system. A function $f$ is called
\textit{additive} if $f(0)=0$ and
\begin{gather*}
  f(g)\equiv\sum_{h=0}^\ell f(a_h(g)X^h)\pmod p\quad\text{for}\quad
  g\equiv\sum_{h=0}^\ell a_h(g)X^h\pmod p.
\end{gather*}
\end{defi}

Since we have defined the analogues of number systems and additive functions to
the definitions for number fields above, we now need to extend the length
estimation of Lemma \ref{kv:thm} in order to successfully state the result. But
before we start we need a little linear algebra. We fix a number system
$(p,\mathcal{N})$ and factor $p$ by
\[
p:=\prod_{i=1}^tp_i^{m_i}
\]
with $p_i\in\Z[X]$ irreducible and $\deg p_i=n_i$. Furthermore we
denote by $\beta_{ik}$ the roots of $p_i$ for $i=1,\ldots,t$ and
$k=1,\ldots,n_i$.

Then we define by
\[
\mathcal{R}:=\Z[X]/(p)=\bigoplus_{i=1}^t\mathcal{R}_i\quad\text{with}\quad
\mathcal{R}_i=\Z[X]/\left(p_i^{m_i}\right)
\]
for $i=1,\ldots,t$ the $\Z$-module under consideration and in the same manner by
\[
\mathcal{K}:=\Q[X]/(p)=\bigoplus_{i=1}^t\mathcal{K}_i\quad\text{with}\quad
\mathcal{K}_i=\Q[X]/\left(p_i^{m_i}\right)
\]
for $i=1,\ldots,t$ the corresponding vector space. Finally we denote by $\overline{\mathcal{K}}$ the
completion of $\mathcal{K}$ according to the usual Euclidean distance.

Obviously $\mathcal{R}$ is a free $\Z$-module of rank $n$. Let
$\lambda:\mathcal{R}\to\mathcal{R}$ be a linear mapping and
$\{z_1,\ldots,z_n\}$ be any basis of $\mathcal{R}$. Then
\[
\lambda(z_j)=\sum_{i=1}^na_{ij}z_i\quad(j=1,\ldots,n)
\]
with $a_{ij}\in\Z$. The matrix $M(\lambda)=\left(a_{ij}\right)$ is
called the matrix of $\lambda$ with respect to the basis
$\{z_1,\ldots,z_n\}$. For an element $r\in\mathcal{R}$ we define by
$\lambda_r:\mathcal{R}\to\mathcal{R}$ the mapping of multiplication by
$r$; that is $\lambda_r(z)=rz$ for every $z\in\mathcal{R}$. Then we
define the norm $\mathrm{N}(r)$ and the trace $\mathrm{Tr}(r)$ of an
element $r\in\mathcal{R}$ as the determinant and the trace of
$M(\lambda_r)$, respectively, \textit{i.e.},
\[
\mathrm{N}(r):=\det(M(\lambda_r)),\quad
\mathrm{Tr}(r):=\mathrm{Tr}(M(\lambda_r)).
\]
Note that these are unique despite of the used basis $\{z_1,\ldots,z_n\}$.
We can canonically extend these notions to $\mathcal{K}$ and
$\overline{\mathcal{K}}$ by everywhere replacing $\Z$ by $\Q$ and $\R$,
respectively.

In the following we will need parameters which help us bounding the length of
the expansion of an element $g\in\mathcal{R}$. Therefore let $g\in\Z[X]$ be a
polynomial, then we put
\[
B_{ijk}(g):=\left.\frac{\mathrm{d}^{j-1}
    g}{\mathrm{d}X^{j-1}}\right\vert_{X=\beta_{ik}}\quad
(i=1,\ldots,t;j=1,\ldots,m_i;k=1,\ldots,n_i).
\]
In
connection with these values we define the ``house'' function H as
\begin{gather*}
  H(g):=\max_{i=1}^t\max_{j=1}^{m_i}\max_{k=1}^{n_i}\lvert B_{ijk}(g)\rvert.
\end{gather*}

We want to investigate the elements with bounded maximum length of
expansion. To this end we need a proposition which estimates the
length of expansion in connection with properties of the number
itself. The proof of this proposition will be presented in the following
section.

\begin{prop}\label{prop:expansion:length:estimation}
  Assume that $(p,\mathcal{N})$ is a number system.  Let $N =
  \max\{\lvert a\rvert:a\in\mathcal{N}\}$ and we set
  \[
  M(g):=\max\left\{\frac{\log\lvert
      B_{ijk}(g)\rvert}{\log\lvert\beta_{ik}\rvert}:i=1,\dots,t;j=1,\ldots,m_i;k=1,\ldots,n_i\right\}.
  \]
  If $g\in \Z[X]$ is of degree at most $n-1$, then for any
  $\varepsilon>0$ there exists $L = L(\varepsilon)$ such that if
  $\ell(g) > L$ then
  \begin{gather}
    \lvert\ell(g)-M(g)\rvert\leq C.
  \end{gather}
\end{prop}

This provides us with an estimation for the length of the expansion and
motivates us to look at subsets of $\mathcal{R}$ where the absolute values
$B_{ijk}$ are bounded. For a vector
$\mathbf{T}:=(T_1,\ldots,T_n)=(T_{111},\ldots,T_{11n_1},T_{121},\ldots,
T_{1,m_i,n_1},T_{211},\ldots,T_{t,m_t,n_t})$ we denote by
\begin{gather}
  \mathcal{R}(\mathbf{T}):=\left\{g\in\mathcal{R}:\lvert
    B_{ijk}(g)\rvert\leq T_{ijk}\right\}\\
  \mathcal{R}_i(\mathbf{T}):=\left\{g\in\mathcal{R}_i:\lvert
    B_{ijk}(g)\rvert\leq T_{ijk}\right\}.
\end{gather}
We want to let the length of expansion to smoothly increase. Therefore
we fix a $T$ and set $T_{ijk}$ for $i=1,\ldots,t$, $j=1,\ldots,m_i$,
$k=1,\ldots,n_i$ such that
\begin{gather}\label{mani:T_ikj}
  \log T_{ijk}=\log T\frac{\log\lvert\beta_{ik}\rvert^n}
  {\log\prod_{i=1}^t\prod_{k=1}^{n_i}\lvert\beta_{ik}\rvert^{m_i}}.
\end{gather}
Remark that $T_{ijk}$ is independent from $j$, which will be important
in Lemma \ref{lem:num:elements}. In view of Proposition
\ref{prop:expansion:length:estimation} we get that the expansions of the
elements in $\mathcal{R}(\mathbf{T})$ almost have the same maximum length. If
not stated otherwise we denote by $\mathbf{T}$ the vector
$(T_{111},\ldots,T_{t,m_t,n_t})$ where the $T_{ijk}$ are as in
(\ref{mani:T_ikj}).

Since $X$ is an invertible element in $\mathcal{K}$ we may extend the
definition of a number system for negative
powers of $X$. Then for $\gamma\in\overline{\mathcal{K}}$ such that
\[
\gamma=\sum_{h=-\infty}^\ell a_hX^h\quad\text{with}\quad
a_h\in\mathcal{N}
\]
we call
\[
\lfloor\gamma\rfloor:=\sum_{h=0}^\ell a_hX^h\quad\text{and}\quad
\{\gamma\}:=\sum_{h=-\infty}^{-1} a_hX^h
\]
the \textit{integer part} and \textit{fractional part} of $\gamma$,
respectively.

Now we have collected all the tools to state our main result.
\begin{thm}\label{thm:result}
  Let $(p,\mathcal{N})$ be a number system and $f$ be an additive
  function such that $f(aX^h)=\mathcal{O}(1)$ as $h\to\infty$ and
  $a\in\mathcal{N}$. Furthermore let
  \begin{align*}
    m_h:=\frac1{\lvert\mathcal{N}\rvert}\sum_{a\in\mathcal{N}}f(aX^h),\quad
    \sigma_h^2:=\frac1{\lvert\mathcal{N}\rvert}\sum_{a\in\mathcal{N}}f^2(aX^h)-m_h^2,
  \end{align*}
  and
  \begin{align*}  
    M(x):=\sum_{h=0}^Lm_h,\quad
    D^2(x)&:=\sum_{h=0}^L\sigma_h^2,
  \end{align*}
  where $L=\lfloor\log_{p(0)}x\rfloor$. Assume that there exists an $\varepsilon>0$ such that
  $D(x)/(\log x)^\varepsilon\to\infty$ as $x\to\infty$ and let
  $P\in\overline{\mathcal{K}}[Y]$ be a polynomial of degree $d$. Then, as
  $T\to\infty$ let $T_{ijk}$ be as in (\ref{mani:T_ikj}),
  \[
  \frac1{\#\mathcal{R}(T)}\#\left\{ z\in\mathcal{R}(T):\frac{f(\lfloor
      P(z)\rfloor)-M(T^d)}{D(T^d)}<y
  \right\}\to\frac1{\sqrt{2\pi}}\int_{-\infty}^y\exp(-x^2)\mathrm{d}x.
  \]
\end{thm}

Our theorem shows that the distribution properties of patterns in the sequence
of digits depend neither on the polynomial $p$ nor on the quotient ring
$\mathcal R$. They are intimate properties of the "backward division algorithm"
defined in \cite{pethoe1991:polynomial_transformation_and}.

We will show the main theorem in several steps.
\begin{enumerate}
\item In the following section we will show properties of number systems which
  we need on the one hand to estimate the length of expansion and on the other
  hand to provide us with an Urysohn function, that helps us counting the
  occurrences of a fixed pattern of digits in the expansion.
\item Equipped with these tools we will estimate the exponential sums occurring in the
  proof in Section \ref{sec:estimation-weyl-sum}. Therefore we need to split
  the module $\mathcal{R}$ up into its components and consider each of them
  separately. We also show that we may neglect the nilpotent elements.
\item Now we take a closer look at the Urysohn function, which will count the
  occurences of our pattern in the expansion, and estimate the number of hits
  of the border of this function in Section \ref{sec:treatment-border}. In
  particular, we count the number of hits of the area, where the function value
  lies between 0 and 1 as this area corresponds to the error term.
\item In Section \ref{sec:main-proposition} we will show that any chosen
  patterns of digit and position occurs uniformly in the expansions. This
  will be our central tool in the proof of Theorem \ref{thm:result}.
\item Finally we draw all the thinks together. The main idea here is to use the
  growth rate of the deviation together with the Fr{\'e}chet-Shohat Theorem to
  cut of the head and the tail of the expansion. Then an application of the
  central Proposition \ref{mani:prop:border} and juxtaposition of the moments
  will prove the result.
\end{enumerate}

\section{Number system properties}\label{sec:number:system}
In this section we want to show two properties we need in the
sequel. The first deals with the above mentioned
estimation of the length of an expansion (Proposition
\ref{prop:expansion:length:estimation}). We will need this result in order to
justify our choice of $\mathbf{T}$ as in (\ref{mani:T_ikj}). Secondly we
construct the Urysohn function for indicating the elements starting with a
certain digit. The main idea is to embed the elements of $\mathcal{R}$ in
$\R^n$ and to use the properties of matrix number systems in this field.

We start with the

\begin{proof}[Proof of Proposition
  \ref{prop:expansion:length:estimation}]
In the proof we combine ideas from
\cite{akiyama_rao2004:new_criteria_canonical} and
\cite{Kovacs_Petho1992:representation_algebraic_integers}.

We may assume $g\not= 0$. As $\{p,{\mathcal N}\} $ is a number system there are
$\ell=\ell(A)$ and $a_h\in {\mathcal N}$ for $h=0,\dots,\ell$; $a_{\ell}\not= 0$ such that
$$
g \equiv \sum_{h=0}^{\ell} a_h X^h \pmod{p},
$$
i.e.,
$$
g = \sum_{h=0}^{\ell} a_h X^h + r\,p
$$
with a polynomial $r\in \Z[X]$. For $j\ge 1$ this implies
\begin{equation} \label{eq:2}
\frac{\mathrm{d}^{j-1} g}{\mathrm{d}X^{j-1}}
=\sum_{h=j-1}^{\ell}\frac{h!}{(h-j+1)!}a_h X^{h-j+1}
  +\sum_{s=0}^{j-1}\binom{j-1}s\frac{\mathrm{d}^s r}{\mathrm{d}X^s}\frac{\mathrm{d}^{j-1-s} p}{\mathrm{d}X^{j-1-s}}.
\end{equation}
Consider a zero $\beta_{ik}$ of $p$, which has multiplicity $m_i$. As we noticed in the Introduction, the argument of the proof of Theorem 6.1. of \cite{pethoe1991:polynomial_transformation_and} allows to prove that $|\beta_{ik}|> 1$ for all $i=1,\dots,t; k=1,\dots,n_i$. Inserting $\beta_{ik}$ into (\ref{eq:2}) we obtain
$$
B_{ijk}(g)=\left.\frac{\mathrm{d}^{j-1}g}{\mathrm{d}X^{j-1}}\right\vert_{X=\beta_{ik}}
=\sum_{h=j-1}^{\ell} \frac{h!}{(h-j+1)!}a_h \beta_{ik}^{h-j+1}
$$
for $i=1,\ldots,t$ and $j=1,\ldots, m_i$. This implies by taking absolute value
\begin{eqnarray*}
\lvert B_{ijk}(g)\rvert &\le& N \sum_{h=j-1}^{\ell} \frac{h!}{(h-j+1)!} |\beta_{ik}|^{h-j+1} \\
&\le& N \frac{\ell!}{(\ell-j+1)!}|\beta_k^{\ell-j+1}| \sum_{h=j-1}^{\ell} \frac{h(h-1)\dots(h-j+1)}{\ell(\ell-1)\dots(\ell-j+1)}|\beta_{ik}|^{h-\ell} \\
&\le& N\frac{\ell^{j-1}|\beta_{ik}|^{\ell}}{|\beta_{ik}|-1},
\end{eqnarray*}
which verifies the lower bound for $\ell$, because $|\beta_{ik}|> 1$.

Now we turn to prove the upper bound. Denote by $V=V_p$ the following mapping:
for $g\in \Z[X]$ of degree at most $n-1$ choose an $a\in {\mathcal N}$ such
that $g(0) \equiv a \pmod{p(0)}$. Such an $a$ exists by Theorem 6.1 of
\cite{pethoe1991:polynomial_transformation_and}. Putting $q=\frac{g(0)-a}{p(0)}$,
let $V(g)=\frac{g-q\cdot p-a}{X}$. Obviously $V(g)\in\Z[X]$ and has degree at
most $n-1$, thus $V$ can be iterated. Moreover we have
\begin{equation} \label{eq:3}
g \equiv \sum_{h=0}^{u} a_h X^h + X^{u+1} V^{u+1}(g) \pmod{p}
\end{equation}
with $a_h \in \mathcal N$ for $h=0,\ldots,u$.

Choose $u$ the largest integer satisfying $|B_{ijk}(g)|\ge
\frac{u^j|\beta_{ik}|^{u-j+1}}{|\beta_{ik}|-1}$ for all $i=1,\dots,t$,
$j=1,\dots,m_i$ and $k=1,\ldots,n_i$. Then $u\le
(1+\varepsilon/2)M(A)$. Proceeding like in the previous case we get
\begin{eqnarray*}
B_{ijk}(g)=\left.\frac{\mathrm{d}^{j-1}g}{\mathrm{d}X^{j-1}}\right\vert_{X=\beta_{ik}}
&=&\sum_{h=j}^{u}\frac{h!}{(h-j+1)!}a_h\beta_{ik}^{h-j}\\
&+&\sum_{s=0}^{j}\binom{j-1}s\frac{(u+1)!}{(u+1-s)!}\beta_{ik}^{u+1-s}\left.\frac{\mathrm{d}^{j-s-1}
  V^{u+1}(g)}{\mathrm{d}X^{j-s-1}}\right\vert_{X=\beta_{ik}}.
\end{eqnarray*}
By its definition $V^{u+1}(g)$ has integer coefficients. Dividing the last
equation by $\frac{(u+1)!}{(u+1-j)!} \beta_{ik}^{u+1-j}$ and consider the obtained
equations for $i=1,\dots,t$ and $k=1,\ldots,n_i$ and for fixed $i$ and $k$ for
$j=1,\dots,m_i$ successively, then using the choice of $u$ and that
$|\beta_{ik}|>1$ we conclude that
\[\left.\frac{\mathrm{d}^{j-s-1}
    V^{u+1}(g)}{\mathrm{d}X^{j-s-1}}\right\vert_{X=\beta_{ik}} < c,
\]
where $c$ is a constant depending only on $N$ as well as the size and the
multiplicities of the zeroes of $p$. These can be considered as $n$ inequalities for the $n$
unknown coefficients of $V^{u+1}(g)$. Furthermore the determinant of the
coefficient matrix is not zero
(c.f. \cite{akiyama_rao2004:new_criteria_canonical}). Thus the solutions are
bounded. As they are integers there are only finitely many possibilities for
$V^{u+1}(g)$. As $\{p,{\mathcal N}\}$ is a number system, $V^{u+1}(g)$ has a
representation, which length is bounded by a constant, say $c_1$, which depends
only on $N$ as well as the size and the multiplicities of the zeroes of
$p$. Thus $\ell(g)\le u + c_1 \le (1+\varepsilon)M(g)$ and the proposition is
proved.
\end{proof}

Now we turn our attention back to the counting of the numbers and in particular
to the construction of the Urysohn function. In order to properly count the
elements we need the fundamental domain, which is defined as the set of all
numbers whose integer part is zero. Since this is not so easy to define in this
context we want to consider its embedding in $\R^n$. The main idea is to use
the corresponding matrix of the polynomial $p$ and to use properties of
matrix number systems. This idea essentially goes back to Gr\"ochenig
and Haas~\cite{mr1348740}. The following definitions are standard in
that area and we mainly follow Gittenberger and
Thuswaldner~\cite{mr1796518} and
Madritsch~\cite{madritsch2010:asymptotic_normality_b}.

We note that if $(p,\mathcal{N})$ is a number system then $X$ is a
integral power base of $\overline{\mathcal{K}}$, \textit{i.e.},
$\{1,X,\ldots,X^{n-1}\}$ is an $\R$-basis for $\overline{\mathcal{K}}$. Thus we
define the embedding $\phi$ by
\[
\phi:\left\{
  \begin{array}{cccc}
    \overline{\mathcal{K}}&\to&\R^n,\\
    a_1+a_2X+\cdots+a_nX^{n-1}&\mapsto
    &(a_1,\ldots,a_n).
  \end{array}\right.
\]

Now let $p=b_{n-1}X^{n-1}+\cdots+b_1X+b_0$. Then we define the corresponding matrix $B$ by
\begin{gather}\label{mani:B}
  B=\left(
    \begin{array}{cccccc}
      0      & 0      & \cdots & \cdots & \cdots & -b_0   \\
      1      & 0      & \cdots & \cdots & 0      & \vdots \\
      0      & 1      & \ddots &        & \vdots & \vdots \\
      \vdots & \ddots & \ddots & \ddots & \vdots & \vdots \\
      \vdots &        & \ddots & 1      & 0      & \vdots \\
      0      & 0      & \cdots & 0      & 1      & -b_{n-1}\\
    \end{array}\right).
\end{gather}

One easily checks that $\phi(X\cdot A)=B\cdot\phi(A)$. Since $B$ is
invertible we can extend the definition of $\phi$ by setting for an integer
$h$
\begin{gather}\label{mani:exchangeXB}
\phi(X^h\cdot A):=B^h\phi(A).
\end{gather}

By this we define the (embedded) fundamental domain by
\[
\mathcal{F}:=\left\{z\in\R^n\middle\vert z=\sum_{h\geq1}B^{-h}a_h,
  a_h\in\phi(\mathcal{N})\right\}.
\]

Following Gr\"ochenig and Haas~\cite{mr1348740} we get that
\[
\lambda((\mathcal{F}+g_1)\cap(\mathcal{F}+g_2))=0
\]
for every $g_1,g_2\in\Z^n$ with $g_1\neq g_2$, where $\lambda$ denotes
the $n$ dimensional Lebesgue measure. Thus $(B,\phi(\mathcal{N}))$ is
a matrix number system and a so called \textit{just touching covering
  system}. Therefore we are allowed to apply the results of the paper by
M{\"u}ller \text{et al.} \cite{mueller01:_fract_proper_of_number_system}.

We now follow the lines of Madritsch
\cite{madritsch2010:asymptotic_normality_b} where the ideas of
Gittenberger and Thuswaldner \cite{mr1796518} were combined with the
results of K\'atai and K\"ornyei \cite{MR1189110} and M\"uller
\text{et al.}~\cite{mueller01:_fract_proper_of_number_system}.

Our main interest is the fundamental domain consisting of all numbers
whose first digit equals $a\in\mathcal{N}$, \textit{i.e.},
\[
\mathcal{F}_a=B^{-1}(\mathcal{F}+\phi(a)).
\]
Imitating the proof of Lemma 3.1 of \cite{mr1796518} we get the
following.
\begin{lem}\label{gt:lem3.1}
  For all $a\in\mathcal{N}$ and all $v\in\N$ there exist a
  $1\leq\mu<\lvert\det B\rvert$ and an axe-parallel tube $P_{v,a}$
  with the following properties:
  \begin{itemize}
  \item $\partial\mathcal{F}_a\subset P_{v,a}$ for all $v\in\N$,
  \item the Lebesgue measure of $P_{v,a}$ is a
    $\mathcal{O}\left(\frac{\mu^v}{\lvert\det B\rvert^v}\right)$,
  \item $P_{v,a}$ consists of $\mathcal{O}(\mu^v)$ axe-parallel
    rectangles, each of which has Lebesgue measure
    $\mathcal{O}(\lvert\det B\rvert^v)$,
  \end{itemize}
  where $\lambda$ denotes the Lebesgue measure.
\end{lem}

As in the proof of Lemma 3.1 of \cite{mr1796518} we can construct for
each pair $(v,a)$ an axe-parallel polygon $\Pi_{v,a}$ and the
corresponding tube
\[
P_{v,a}:=\left\{z\in\R^n\middle\vert\lVert z-\Pi_{v,a}\rVert_\infty
  \leq2c_p\lvert\det B\rvert^{-v}\right\},
\]
where $c_p$ is an arbitrary constant.
Furthermore we denote by $I_{v,a}$ the set of all points inside
$\Pi_{v,a}$. Now we define our Urysohn function $u_a$ by
\begin{gather}\label{mani:urysohn}
  u_a(x_1,\ldots,x_n)=\frac1{\kappa^n}\int_{-\frac\kappa2}^{\frac\kappa2}\cdots
  \int_{-\frac\kappa2}^{\frac\kappa2}\psi_a(x_1+y_1,\ldots,x_n+y_n)
  \,\mathrm{d}y_1\cdots\mathrm{d}y_n,
\end{gather}
where
\begin{gather}\label{mani:Delta}
  \kappa:=2c_u\lvert\det B\rvert^{-v}
\end{gather}
with $c_u$ a constant and
\[
\psi_a(x_1,\ldots,x_n)=\begin{cases}
  1&\text{if }(x_1,\ldots,x_n)\in I_{v,a}\\
  \frac12&\text{if }(x_1,\ldots,x_n)\in \Pi_{v,a}\\
  0&\text{otherwise.}\end{cases}
\]
Thus $u_a$ is the desired Urysohn function which equals $1$ for $z\in
I_{v,a}\setminus P_{v,a}$, $0$ for $z\in\R^n\setminus(I_{v,a}\cup
P_{v,a})$, and linear interpolation in between.

We now do a Fourier transform of $u_a$ and estimate the coefficients
in the same way as in Lemma~3.2 of \cite{mr1796518}.

\begin{lem}\label{gt:lem3.2}
  Let $u_a(x_1,\ldots,x_n)=\sum_{(m_1,\ldots,m_n)\in\Z^n}
  c_{m_1,\ldots,m_n} e(m_1x_1+\cdots+m_nx_n)$ be the Fourier series of
  $u_a$. Then the Fourier coefficients $c_{m_1,\ldots,m_n}$ can be
  estimated by
  \begin{gather*}
    c_{0,\ldots,0}=\frac1{\lvert\det B\rvert},\quad
    c_{m_1,\ldots,m_n}\ll\mu^v\prod_{i=1}^n\frac1{r(m_i)}
  \end{gather*}
  with
  \[
  r(m_i)=\begin{cases}\kappa m_i&m_i\neq0,\\
    1&m_i=0.\end{cases}
  \]
\end{lem}

\section{Estimation of the Weyl Sum}\label{sec:estimation-weyl-sum}
Before we continue with the estimation of the number of points inside the
fundamental domain and those hitting the boarder, we want to estimate
the exponential sums, which will occur in the following sections. In
particular we want to prove the following.
\begin{prop}\label{mani:prop:weyl}
  Let $T\geq0$ and $T_{ijk}$ as in   (\ref{mani:T_ikj}). Let $L$ be the maximum
  length of the $b$-adic   expansion of $z\in {\mathcal R}(\mathbf{T})$ and let $C_1$ and
  $C_2$ be sufficiently large constants. Furthermore let $l_1,\ldots,l_h$ be
  positions and $\mathbf{h}_1,\ldots,\mathbf{h}_h$ be corresponding
  $n$-dimensional vectors. If
  \begin{gather}\label{mani:weyl:l_bounds}
    C_1\log L\leq l_1< l_2<\cdots<l_h\leq dL-C_2\log L
  \end{gather}
  and
  \begin{gather}\label{mani:weyl:h_bounds}
    \lVert\mathbf{h}_r\rVert_\infty\leq(\log T)^{\sigma_1}
  \end{gather}
  for $1\leq r\leq h$, then we have
  \[
  \sum_{z\in\mathcal{R}(\mathbf{T})}e\left(\sum_{r=1}^h\left\langle\mathbf{h}_r,B^{-l_r-1}\phi(P(z))\right\rangle\right)\ll
  T^n(\log T)^{-t\sigma_0}
  \]
where $\sigma_0$ depends on $\sigma_1$, $C_1$ and $C_2$.
\end{prop}

Our main idea consists in several steps. First we will split the ring
$\mathcal{R}$ up into the $\mathcal{R}_i$ and consider each of them
separately. Then we distinguish two cases according to whether $m_i>1$ or
not. The latter reduces to an estimation of the sum in an algebraic number
field. Whereas for the case of $m_i>1$ we have to deal with nilpotent
elements. Therefore we divide $\mathcal{R}_i$ into the radical and the nilpotent elements. Thus we define $\tilde{\mathcal{R}_i}$ as
\begin{gather}\label{mani:Rradical}
\tilde{\mathcal{R}_i}:=\Z[X]/(p_i)\quad\text{and}\quad
\tilde{\mathcal{R}_i}(\mathbf{T}):=
  \left\{g\in\tilde{\mathcal{R}_i}: \lvert B_{ijk}(g)\rvert\leq T_{ijk}\right\}
\end{gather}
and the set $\mathcal{N}_i$ to be the nilpotent elments, \textit{i.e.},
\begin{gather}\label{mani:Ni}
\mathcal{N}_i:=\left\{g\in\mathcal{R}_i:g\equiv 0\bmod
  p_i\right\}\quad\text{and}\quad
\mathcal{N}_i(\mathbf{T}):=\left\{g\in\mathcal{R}_i(\mathbf{T}):g\equiv 0\bmod
  p_i\right\}.
\end{gather}

But before we start with the proof we need to show that the estimation
is good compared with the trivial one. Thus we will show the
following.
\begin{lem}\label{lem:num:elements}
Let $T_{ijk}$ for $i=1,\ldots,t$, $j=1,\ldots,m_i$, $k=1,\ldots,n_i$
be positive reals. Then
\begin{align*}
  \#\mathcal{R}(\mathbf{T})&=\prod_{i=1}^t\#\mathcal{R}_i(\mathbf{T}),\\
  \#\mathcal{R}_i(\mathbf{T})&=c_i\left(\prod_{k=1}^{n_i}T_{i1k}\right)^{m_i}
  +\mathcal{O}\left(T_0^{m_in_i-1}\right),\\
  \#\mathcal{N}_i(\mathbf{T})&=c_i\left(\prod_{k=1}^{n_i}T_{i1k}\right)^{m_i-1}+
  \mathcal{O}\left(T_0^{(m_i-1)n_i-1}\right),
\end{align*}
where
\[
  T_0=\max_{i=1}^t\max(1,(T_{i11}\cdots T_{i,m_i,n_i})^{\frac1{m_in_i}})
\]
and the constants $c_i$ will be defined in Lemma~\ref{madritsch:lem3.3}.
\end{lem}

\begin{proof}
The first assertion follows immediately from the definition of
$\mathcal{R}(\mathbf{T})$. Since the $\mathcal{R}_i$ are independent
we fix an $i$ and focus on $\mathcal{R}_i(\mathbf{T})$. Obviously we
have that
\begin{gather}\label{mani:isomorph}
\mathcal{R}_i=\Z[X]/(p_i^{m_i})\cong\left(\Z[X]/(p_i)\right)^{m_i}.
\end{gather}
Thus we concentrate on $\Z[X]/(p_i\Z[X])$ which is an order in a number field
$\mathcal{K}_i$ of degree $n_i$ over $\Q$. For $\gamma\in\mathcal{K}_i$ let
$\gamma^{(\ell)}$ ($1\leq\ell\leq r_1$) be the real conjugates and
$\gamma^{(m)}$ and $\gamma^{(m+r_2)}$ ($r_1+1\leq m\leq r_1+r_2$) be the pairs
of complex conjugates of $\gamma$. Note that $r_1+2r_2 = n_i$. We will apply
the following lemma.
\begin{lem}[{\cite[Lemma
    3.3]{madritsch2010:asymptotic_normality_b}}]\label{madritsch:lem3.3}
Let $T_k$ ($1\leq k\leq r_1+r_2$) be positive integers and set
$T_{r_1+r_2+k}=T_{r_1+k}$ for $1\leq k\leq r_2$. Then
\[
  \#\left\{a\in\Z[X]/(p_i):\lvert a^{(k)}\rvert\leq T_k\right\}
  =c_iT_{1}\cdots T_{n_i}
  +\mathcal{O}\left(T_0^{n_i-1}\right),
\]
where $T_0=\max\left(1,(T_{1}\cdots T_{n_i})^{1/n_i}\right)$ and $c_i$ is
a constant depending on $\Z[X]/(p_i)$.
\end{lem}

Furthermore since \eqref{mani:isomorph} holds, we get that there exists a $\Z$
linear mapping $M_i$ such that
\[
  M_i\cdot(T_{i11},\ldots,T_{i,m_i,n_i})=
  (\tilde{T}_{i11},\ldots,\tilde{T}_{i,m_i,n_i})
\]
and
\[
  \#\mathcal{R}_i(\mathbf{T})=\prod_{j=1}^{m_i}\#\left\{a\in\Z[X]/(p_i)
  :\lvert a^{(k)}\rvert\leq\tilde{T}_{ijk}\quad 1\leq k\leq n_i\right\}.
\]

As the value of $T_{ijk}$ is independent from $j$, an application of
Lemma \ref{madritsch:lem3.3} yields
\[
  \#\mathcal{R}_i(\mathbf{T})=\tilde{c}_i\left(\prod_{k=1}^{n_i}T_{i1k}\right)^{m_i}
  +\mathcal{O}\left(T_{i0}^{m_in_i-1}\right),
\]
where $\tilde{c}_i$ depends on $c_i$ and $M_i$ and
\[
  T_{i0}=\max\left(1,\left(T_{i11}\cdots T_{i,m_i,n_i}\right)^{\frac1{m_in_i}}\right).
\]

For the estimate involving $\mathcal{N}_i(\mathbf{T})$ we note that
\[
\mathcal{N}_i=\left\{g\in\mathcal{R}_i:g\equiv 0\bmod p_i\right\}
\cong\left(\Z[X]/(p_i)\right)^{m_i-1}
\]
and the result follows in the same way as for $\mathcal{R}_i(\mathbf{T})$.
\end{proof}

With help of all these tools we can show Proposition \ref{mani:prop:weyl}.

\begin{proof}[Proof of Proposition \ref{mani:prop:weyl}]
  The first step consists in splitting the sum over $\mathcal{R}(\mathbf{T})$
  up into those over $\mathcal{R}_i(\mathbf{T})$. Therefore let
  $\pi_i:\mathcal{R}\to\mathcal{R}_i$ be the canonical
  projections. Then $\pi:=(\pi_1,\ldots,\pi_t)$ is an isomorphism by
  the Chinese Remainder Theorem. Furthermore let $\phi_i$ be the
  embedding defined by
  \[
  \phi_i:\left\{
    \begin{array}{cccc}
      \overline{\mathcal{K}_i}&\to&\R^{n_im_i},\\
      a_1+a_2X+\cdots+a_{m_in_i}X^{n_im_i-1}&\mapsto
      &(a_1,\ldots,a_{m_in_i})
    \end{array}\right.
  \]
  for $i=1,\ldots,t$. Finally we define the matrix $M$ to be such that
  \[
  M\cdot\phi(z):=(\phi_1\circ\pi_1(z),\ldots,\phi_t\circ\pi_t(z)).
  \]
  Furthermore we note that for $P_i:=\pi_i\circ P$ and $l\in\Z$ we have
  \[
    M\cdot\phi(P(z)X^l)=(\phi_1(P_1(z_1)X^l),\ldots,\phi_t(P(z_t)X^l)).
  \]

  Thus
  \begin{align*}
    \sum_{z\in\mathcal{R}(\mathbf{T})}e\left(\sum_{r=1}^h\left\langle\mathbf{h}_r,\phi(P(z)X^{-l_r-1}) \right\rangle\right)
    =\prod_{i=1}^t\sum_{z_i\in\mathcal{R}_i(\mathbf{T})}e\left(\sum_{r=1}^h\left\langle\mathbf{h}_{ri}, \phi_i(P_i(z_i)X^{-l_r-1})\right\rangle\right)
  \end{align*}
  where $(\mathbf{h}_{r1},\ldots,\mathbf{h}_{rt}):=\mathbf{h}_rM^{-1}$.

  Now we will consider each sum over $\mathcal{R}_i(\mathbf{T})$
  separately. Therefore we fix until the end of the proof an $1\leq i\leq t$
  and distinguish two cases according to whether $m_i=1$ or not.
  \begin{itemize}
  \item\textbf{Case 1,} $m_i=1$: In this case we set $\beta=\beta_{i1}$ and
    observe that $K=\mathcal{K}_i=\Q(\beta_{i1})$ and
    $\mathcal{R}_i\cong\Z[\beta]$. Furthermore let $\mathcal{O}_K$ be the
    maximum order aka the ring of integers of $K$, then clearly
    $\Z[\beta_{i1}]\subset\mathcal{O}_K$. We denote by $\beta_{ik}=\beta^{(k)}$
    the conjugates of $\beta$. Now we will proceed as in the proof of
    Proposition 6.1 of  Madritsch~\cite{madritsch2010:asymptotic_normality_b}.

    Therefore we need some parameters of the field $K$ and for
    short we set $n=n_i$ during this case. Then we order the conjugates by
    denoting with $\beta^{(k)}$ for $1\leq k\leq r_1$ the real conjugates,
    whereas $\beta^{(k)}$ and $\beta^{(k+r_2)}$ denote the pairs of complex
    conjugates, where $n=r_1+2r_2$. Let $\Tr$ be the trace of an element of
    $K$ over $\Q$, then we define
    \begin{gather}\label{mani:tau}
      \tau(z):=\left(\Tr(z),\Tr(\beta
        z),\ldots,\Tr(\beta^{n-1}z)\right) =\Xi\phi_i(z),
    \end{gather}
    where $\Xi=VV^T$ and $V$ is the Vandermonde matrix
    \[
    V=\left(\begin{array}{cccc}
        1 & 1 & \cdots & 1 \\
        \beta & \beta^{(2)} & \cdots & (\beta^{(n)})^{n-1}\\
        \vdots & \vdots &    & \vdots\\
        \beta^{n-1} & (\beta^{(2)})^{n-1} & \cdots & (\beta^{(n)})^{n-1}\\
      \end{array}\right).
    \]

    We set $(\tilde{h}_{r1},\ldots,\tilde{h}_{rn}):=\mathbf{h}_{ri}\Xi^{-1}$
    and note that
    \[
    \left\langle\mathbf{h}_{ri},\phi_i(P_i(z_i)X^{-l_r-1})\right\rangle
    =\mathbf{h}_{ri}^T\Xi^{-1}\tau(P_i(z_i)\beta^{-l_r-1})
    =\Tr\left(\sum_{k=1}^n\tilde{h}_{rk}\beta^{k-l_r-2}P_i(z_i)\right).
    \]
    Thus we may rewrite the sum under consideration as follows
    \[
    \sum_{z\in\mathcal{R}_i(\mathbf{T})}e\left(\sum_{r=1}^h\left\langle\mathbf{h}_{ri},\phi_i(P_i(z_i)X^{-l_r-1})\right\rangle\right)=
    \sum_{z\in\mathcal{R}_i(\mathbf{T})}e\left(\Tr\left(\sum_{r=1}^h
        \sum_{k=1}^n\tilde{h}_{rk}\beta^{k-l_r-2}P_i(z_i)\right)\right).
    \]

    Now we need an approximation lemma which essentially goes back to
    Siegel
    \cite{Siegel1944:generalization_waring's_problem}. Therefore let
    $\delta$ be the different of $K$ over $\Q$ and $\Delta$ be the
    absolute value of the discriminant of $K$. Then we have the following.

  \begin{lem}\label{siegel:lem}
    Let $N_1,\ldots,N_{r_1+r_2}$ be real numbers and let $N=\sqrt[n]{N_1\cdots
      N_{r_1}(N_{r_1+1}\cdots N_{r_1+r_2})^2}$ be their geometric
    mean. Suppose that $N>\Delta^{\frac1n}$, then, corresponding to
    any $\xi\in K$, there exist $q\in\mathcal{O}_K$ and
    $a\in\delta^{-1}$ such that
    \begin{gather*}
      \lvert q^{(k)}\xi^{(k)}-a^{(k)}\rvert<N_k^{-1},\quad 0<\lvert
      q^{(k)}\rvert\leq N_k,\quad 1\leq k\leq r_1+r_2,\\
      \max\left(N_k\lvert q^{(k)}\xi^{(k)}-a^{(k)}\rvert,\lvert
        q^{(k)}\rvert\right)\geq \Delta^{-\frac12},\quad 1\leq k\leq r_1+r_2,
      \intertext{and} \mathrm{N}((q,a\delta))\leq \Delta^{\frac12}.
    \end{gather*}
  \end{lem}

  For $1\leq r\leq h$ we set $\xi_r$ to be the leading coefficient of
  $\sum_{k=1}^n\tilde{h}_{rk}P_i(z)$. Then we apply Lemma
  \ref{siegel:lem} with $N_k=T_{i,1,k}^{d}(\log T)^{-\sigma_2}$ for $1\leq
  k\leq r_1+r_2$ in order to get that there exist $a\in\delta^{-1}$ and
  $q\in\mathcal{O}_K$ such that
  \[
  \lvert\sum_{r=1}^h\frac{\xi^{(k)}_r}{(\beta^{(k)})^{l_r+1}}\,q^{(k)}-a^{(k)}\rvert
  <\frac{(\log T)^{\sigma_2}}{T_{i,1,k}^d} \quad\text{and}\quad 0<\lvert
  q^{(k)}\rvert<\frac{T_{i,1,k}^d}{(\log T)^{\sigma_2}} \quad\text{for}\quad
  1\leq k\leq n.
  \]

\begin{lem}[{\cite[Proposition 3.2]{madritsch2010:asymptotic_normality_b}}]
\label{madr:prop3.2}
  Suppose that
  \begin{gather}\label{mani:g}
    Q(X)=\alpha_dX^d+\cdots+\alpha_1X
  \end{gather}
  is a polynomial of degree $d$ with coefficients in $K$. If for the leading
  coefficient   $\alpha_d$ there exist $a\in\delta^{-1}$ and
  $q\in\mathcal{O}_K$ as in Lemma \ref{siegel:lem} with $N_k=T_{i1k}^d(\log
  T)^{-\sigma_2}$ and
  \[
  (\log T)^{\sigma_2}\leq\lvert q^{(k)}\rvert\leq T_{i1k}^d(\log
  T)^{-\sigma_2} \quad 1\leq k\leq r_1+r_2,
  \]
  then
  \[
  \sum_{x\in \mathcal{R}_i(\mathbf{T})}e(\Tr(Q(x)))\ll T^{n_i}(\log T)^{-\sigma_0}
  \]
  with $\sigma_2\geq 2^{d-1}\left(\sigma_0+r2^{2d}\right)$.
\end{lem}

Now we distinguish several cases according to the quality of approximation by
Lemma~\ref{siegel:lem}, which is represented by the size of $H(q)$:
\begin{itemize}
\item\textbf{Case 1.1,} $\house{q}\geq(\log T)^{\sigma_2}$: We apply
  Lemma \ref{madr:prop3.2} and get
  \[
  \sum_{z_i\in\mathcal{R}_i(\mathbf{T})}
  E\left(\sum_{r=1}^h\sum_{k=1}^{n}\frac{\tilde{h}_{rk}P_i(z_i))}{\beta^{l_r+1}}\right)
  \ll T^n(\log T)^{-\sigma_0}.
  \]

\item\textbf{Case 1.2,} $2\leq\house{q}<(\log T)^{\sigma_2}$: In the last two
  cases we need Minkowski's lattice theory
  (\textit{cf}. \cite{hua1982:introduction_to_number}). Let $\lambda_1$ be the
  first successive minimum of the $\Z$-lattice $\delta^{-1}$. Then we get
  \begin{gather*}
  \lvert\sum_{r=1}^h\frac{\xi^{(k)}_r}{(\beta^{(k)})^{l_r+1}}\rvert
  \geq\lvert\frac{a^{(k)}}{q^{(k)}}\rvert-\frac1{\lvert q^{(k)}\rvert^2}
  \geq\lambda_1\left(\frac1{\lvert q^{(k)}\rvert}-\frac1{\lvert
      q^{(k)}\rvert^2}\right)
  \geq\lambda_1\frac1{2\lvert q^{(k)}\rvert}
  \gg(\log T)^{-\sigma_2},
  \end{gather*}
  which implies
  \begin{gather*}
    (\log T)^{\sigma_2}
    \ll\lvert\sum_{r=1}^h\frac{\xi^{(k)}_r}{(\beta^{(k)})^{l_r+1}}\rvert
    \leq\frac{\sum_{r=1}^h\lvert\xi^{(k)}_r\rvert}{\lvert \beta^{(k)}\rvert^{l_1+1}}.
  \end{gather*}

  Since $\lvert \beta^{(k)}\rvert>1$ and (\ref{mani:weyl:h_bounds}) we have
\[
  \lvert \beta^{(k)}\rvert^{l_1+1}\ll\lvert\xi^{(k)}\rvert(\log T)^{\sigma_2}
  \ll n(\log T)^{\sigma_2+\sigma_1},
\]
which yields
\[
  l_1\ll\frac{(\sigma_2+\sigma_1)}{\log\lvert \beta^{(k)}\rvert}\log\log T
\]
contradicting the lower bound of $l_1$ for sufficiently large $C_1$ in
(\ref{mani:weyl:l_bounds}).
\item\textbf{Case 1.3,} $0<\house{q}<2$: In this case we will again use
  Minkowski's lattice theory
  (\textit{cf}. \cite{hua1982:introduction_to_number}). Let $\lambda_1$ be the
  first successive minimum of the $\Z$-lattice $\delta^{-1}$, then we have to
  consider two subcases
  \begin{itemize}
  \item\textbf{Case 1.3.1,}
    $\house{\sum_{r=1}^h\frac{\xi_r}{\beta^{l_r+1}}q}
    \geq\frac{\lambda_1}{2}$: Let $1\leq k\leq n$ be such that
    \[
    \frac{\lambda_1}{2}\leq
    \lvert\sum_{r=1}^h\frac{\xi^{(k)}_r}{(\beta^{(k)})^{l_r+1}}q^{(k)}\rvert\leq
    \frac{\sum_{r=1}^h\lvert\xi^{(k)}_r\rvert}{\lvert
      \beta^{(k)}\rvert^{l_1+1}}\lvert q^{(k)}\rvert,
    \]
    then
    \[
    l_1+1\ll \log\log T
    \]
    again contradicts the lower bound of $l_1$ for sufficiently large
    $C_1$ in (\ref{mani:weyl:l_bounds}).
  \item\textbf{Case 1.3.2,}
    $\house{\sum_{r=1}^h\frac{\xi_r}{\beta^{l_r+1}}q}
    <\frac{\lambda_1}{2}$: By Minkowski's theorem (\textit{cf.}
    \cite{hua1982:introduction_to_number}) we get that $a=0$. Thus for $1\leq k\leq n$
    \[
    \lvert\sum_{r=1}^h\frac{\xi^{(k)}_r}{(\beta^{(k)})^{l_r+1}}q^{(k)}\rvert=
    \lvert\frac{1}{(\beta^{(k)})^{l_r+1}}\sum_{r=1}^h\xi^{(k)}_r(\beta^{(k)})^{l_h-l_r}
    q^{(k)}\rvert\leq\frac{(\log T)^{\sigma_2}}{T_{i1k}^d}
    \]
    which implies (taking the norm of the left side)
    \[
    l_h+1\geq nd\log_{\lvert\mathrm{N}(\beta)\rvert}
    T-c(\log\log_{\lvert\mathrm{N}(\beta)\rvert}T)
    \]
    contradicting the upper bound for sufficiently large $C_2$.
  \end{itemize}
\end{itemize}

\item\textbf{Case 2,} $m_i>1$: Now we have to go one step further and
  to take a closer look at $\mathcal{R}_i$. In particular we divide every
  element $z_i\in\mathcal{R}_i$ into its radical and its nilpotent part. We
  fix an element $z\in\mathcal{R}$ and set $z_i:=\pi_i(z)$.

  On the one hand, since $\mathcal{R}_i=\tilde{\mathcal{R}_i}\oplus\mathcal{N}_i$ we have
  for $z_i\in\mathcal{R}_i$ the unique representation
  \begin{gather}\label{mani:zi}
  z_i=z_{i1}+z_{i2}
  \end{gather}
  with $z_{i1}\in\tilde{\mathcal{R}_i}$ and $z_{i2}\in\mathcal{N}_i$. This motivates the definition of the
  linear map $\pi_{ij}$ such that $\pi_{ij}(z):=z_{ij}$ for $i=1,\ldots,t$
  and $j=1,2$.

  On the other hand, since $\mathcal{R}_i\cong(\Z[X]/(p_i))^{m_i}$ we have for
  every $z_{i}\in\mathcal{R}_i$ the unique representation
  \begin{align*}
    z_i=\sum_{j=1}^{m_i}a_{ij}p_i^{j-1}=\sum_{j=1}^{m_i}\sum_{k=1}^{n_i}a_{ijk}X^{k-1}p_i^{j-1}
  \end{align*}
  with $a_{ij}\in\Z[X]$ and $a_{ijk}\in\Z$, respectively. We clearly have
  \[
    \pi_{i1}(z)=\sum_{k=1}^{n_i}a_{i1k}X^{k-1}\quad\text{and}\quad
    \pi_{i2}(z)=\sum_{j=2}^{m_i}\sum_{k=1}^{n_i}a_{ijk}X^{k-1}p_i^{j-1}.
  \]

  Thus we define for $z_i\in\mathcal{R}_i$ the embeddings $\psi_{i1}$
  and $\psi_{i2}$ by
  \begin{align*}
  \psi_{i1}(\pi_{i1}(z_i))=(a_{i11},\ldots,a_{i1n_i})\quad\text{and}\quad
  \psi_{i2}(\pi_{i2}(z_i))=(a_{i21},\ldots,a_{i,m_i,n_i}).
  \end{align*}
  Then there exists an invertible matrix $\tilde{M}_i$ such that
  \[
  \tilde{M}_i(\phi_i\circ\pi_i(z))=(\psi_{i1}\circ\pi_{i1}(z),\psi_{i2}\circ\pi_{i2}(z)).
  \]

  Now we can divide the sum up as follows.
  \begin{align*}
    &\sum_{z_i\in\mathcal{R}_i(\mathbf{T})}e\left(\sum_{r=1}^h\left\langle\mathbf{h}_{ri},\phi_i(P_i(z_i)X^{-l_r-1})\right\rangle\right)\\
    &\quad=\sum_{z_{i1}\in\tilde{\mathcal{R}_i}(\mathbf{T})}\sum_{z_{i2}\in\mathcal{N}_{i}(\mathbf{T})}
    e\left(\sum_{r=1}^h\left\langle\mathbf{h}_{ri},\phi_{i}(P_i(z_{i1}+z_{i2})X^{-l_r-1})\right\rangle\right)\\
    &\quad=\sum_{z_{i1}\in\tilde{\mathcal{R}_i}(\mathbf{T})}e\left(\sum_{r=1}^h\left\langle\mathbf{h}_{ri1},
        \psi_{i1}(P_{i1}(z_{i1})X^{-l_r-1})\right\rangle\right)\sum_{z_{i2}\in\mathcal{N}_{i}(\mathbf{T})}
e\left(\sum_{r=1}^h\left\langle\mathbf{h}_{ri2},\psi_{i2}(P_{i2}(z_{i2})X^{-l_r-1})\right\rangle\right),
  \end{align*}
  where we have set $P_{ij}=\pi_{ij}\circ P$ for $j=1,2$.

  Since for the first sum we have that $m_i=1$ we may follow \textbf{Case 1} above and use Lemma
  \ref{madritsch:lem3.3} for trivially estimating the second one to prove the
  proposition for this case.
\end{itemize}

\end{proof}

\section{Treatment of the border}\label{sec:treatment-border}
In Section \ref{sec:number:system} above we have constructed the Urysohn
function we need in order to properly count the number of elements
within the fundamental domain. In this construction we also used an
axe-parallel tube in order to cover the border of the fundamental
domain. The number of hits of this tube gives rise to the error term
which we will consider in this section.

We fix a positive integer $v$, which will be chosen later, and a real vector
$\mathbf{T}$. Furthermore for $l\geq0$ we define $F_l$ to be the number of hits
of the border of the Urysohn function which is
\begin{gather}\label{mani:Fl}
  F_l:=\#\left\{z\in {\mathcal R}(\mathbf{T})\middle\vert
    B^{-l-1}\phi\left(P(z)\right)
    \in\bigcup_{a\in\mathcal{N}}P_{v,a}\mod B^{-1}\Z^n\right\}.
\end{gather}

As indicated above we are interested in an estimation of $F_l$.
\begin{prop}\label{mani:prop:border}
  Let $\mu<\lvert\det B\rvert$ be as in Section
  \ref{sec:number:system} and $C_1$ and $C_2$ be sufficiently large
  positive reals. Suppose that $l$ is a positive integer such that
  \begin{gather}\label{mani:j_bounds}
    C_1\log\log T\leq l\leq d\log_{\lvert\det B\rvert}T-C_2\log\log T.
  \end{gather}
  Then for any positive $\sigma_3$ we have
  \[
  F_l\ll\mu^vT^n\left(\lvert\det B\rvert^{-v}+(\log
    T)^{-t\sigma_3}\right).
  \]
\end{prop}

In order to estimate $F_l$ we need the Erd\H os-Tur\'an-Koksma
Inequality.
\begin{lem}[{\cite[Theorem 1.21]{mr1470456}}]\label{etk:inequ}
  Let $x_1,\ldots,x_S$ be points in the $n$-dimensional real vector
  space $\R^n$ and $H$ an arbitrary positive integer. Then the
  discrepancy $D_S(x_1,\ldots,x_S)$ fulfills the inequality
  \[
  D_S(x_1,\ldots,x_S)\ll\frac2{H+1}+
  \sum_{0<\lVert\mathbf{h}\rVert_\infty\leq H}\frac1{r(\mathbf{h})}
  \lvert\frac1S\sum_{s=1}^Se(\langle\mathbf{h},x_s\rangle)\rvert,
  \]
  where $\mathbf{h}\in\Z^n$ and
  $r(\mathbf{h})=\prod_{i=1}^n\max(1,\lvert h_i\rvert)$.
\end{lem}

\begin{proof}[Proof of Proposition \ref{mani:prop:border}]
  We want to proceed in three steps. First we subdivide the tube
  $P_{v,a}$ into rectangles in order to apply the Erd\H
  os-Tur\'an-Koksma Inequality in the second step. Finally we put them
  together to gain the desired result.

  Recall that the tube $P_{v,a}$ defined in Lemma
  \ref{gt:lem3.1} consists of a family of rectangles. Let $R_a$ be one
  of them, then we want to estimate
  \[
  F_l(R_a):=\#\left\{z\in {\mathcal R}(\mathbf{T})\middle\vert
    B^{-l-1}\phi\left(P(z)\right) \in\bigcup_{a\in\mathcal{N}}R_a\mod
    B^{-1}\Z^n\right\}.
  \]

  Using the definition of the discrepancy we get that
  \begin{gather}\label{mani:border0}
    F_l(R_a)\ll T^n\left(\lambda(R_a)
      +D_S\left(\left\{B^{-l-1}\phi\left(P(z)\right)\right\}_{z\in\mathcal{R}(\mathbf{T})}\right)\right),
  \end{gather}
  where $\lambda$ is the $n$-dimensional Lebesgue measure and $S$ is the number
  of elements in $\mathcal{R}(\mathbf{T})$. By Lemma \ref{lem:num:elements} we have that
  \begin{gather}\label{mani:border1}
    S=\prod_{i=1}^tc_i\left(\prod_{k=1}^{n_i}T_{i1k}\right)^{m_i}+\mathcal{O}\left(T_0^{n-1}\right).
  \end{gather}

  Now we apply Lemma \ref{etk:inequ} to get
  \begin{gather}\label{mani:border2}
    D_S\left(\left\{B^{-l-1}\phi\left(P(z)\right)\right\}_{z\in\mathcal{R}(\mathbf{T})}\right)
    \ll\frac2{H+1}+\sum_{0<\lVert\mathbf{h}\rVert_\infty\leq H}
    \frac1{r(\mathbf{h})}\lvert\frac1S\sum_{z\in\mathcal{R}(\mathbf{T})}
    e\left(\langle\mathbf{h},B^{-l-1}\phi\left(P(z)\right)\rangle\right)\rvert.
  \end{gather}

  The next step consists in an application of Proposition
  \ref{mani:prop:weyl} which yields
  \begin{gather}\label{mani:border3}
  \lvert\sum_{z\in\mathcal{R}(\mathbf{T})}
    e\left(\langle\mathbf{h},B^{-l-1}\phi\left(P(z)\right)\rangle\right)\rvert
  \ll T^n(\log T)^{-t\sigma_0}
  \end{gather}

  Putting (\ref{mani:border1}), (\ref{mani:border2}), and
  (\ref{mani:border3}) togetter in (\ref{mani:border0}) gives
\begin{align*}
F_l(R_a)
&\ll T^n\lambda(R_a)+\frac{T^n}{(\log T)^{\sigma_1}}+T^n(\log
T)^{-t\sigma_0}\sum_{0<\lVert\mathbf{h}\rVert_\infty\leq H}\frac1{r(\mathbf{h})}\\
&\ll T^n\lambda(R_a)+\frac{T^n}{(\log T)^{\sigma_1}}+T^n(\log
  T)^{-t\sigma_0}(\log\log T)^n.
\end{align*}
Setting $\sigma_1:=t\sigma_0/2$ and summing over all rectangles
$R_a$ yields
\[
  F_l\ll\mu^vT^n\left(\lvert\mathrm{N}(b)\rvert^{-v}+(\log T)^{-t\sigma_0/2}\right).
\]
Finally we set $\sigma_3=t\sigma_0/2$ which proves the proposition.
\end{proof}

\section{The main proposition}\label{sec:main-proposition}
The main idea is to understand the additive function as putting weights on the
digits. Thus if we can show that the digits are uniformly distributed the same
is true for the values of the additive functions. Therefore we look at patterns
in the expansion of $P(z)$. In particular, we count the number of occurrences
of certain digits at certain positions in the expansions.

\begin{prop}\label{mani:centralprop}
  Let $f$ be an additive function. Let $T\geq0$ and $T_{ijk}$ as in (\ref{mani:T_ikj}). Let $L$ be the maximum
  length of the $b$-adic expansion of $z\in {\mathcal R}(\mathbf{T})$ and let $C_1$ and
  $C_2$ be sufficiently large. Then for
  \begin{gather}\label{mani:l_bounds}
    C_1\log L\leq l_1< l_2<\cdots<l_h\leq dL-C_2\log L
  \end{gather}
  we have,
  \begin{align*}
  \Theta&:=\#\left\{z\in {\mathcal R}(\mathbf{T})\middle \vert
    a_{l_r}(f(z))=b_r,r=1,\ldots,h\right\}\\
  &=\frac{c_1\cdots c_t}{\lvert\det B\rvert^h}T^n+\mathcal{O}\left(T^n(\log T)^{-t\sigma_0}\right).
  \end{align*}
  uniformly for $T\to\infty$, where $(l_r,b_r)\in\N\times\mathcal{N}$ are given
  pairs of position and digit and $\sigma_0$ is an arbitrary positive constant.
\end{prop}

\begin{proof}
  We recall our Urysohn function $u_a$ (defined in
  (\ref{mani:urysohn})) and set for $\mathbf{\nu}\in\R^n$
  \[
  t(\mathbf{\nu})=u_{b_1}(B^{-l_1-1}\mathbf{\nu})\cdots u_{b_h}(B^{-l_h-1}\mathbf{\nu}),
  \]
  where $B$ is the matrix defined in \eqref{mani:B}.

  Now we want to apply the Fourier transformation, which we developed
  in Lemma \ref{gt:lem3.2}. Therefore we set
  \[
  \mathcal{M}:=\left\{M=(h_1,\ldots,h_h)\middle\vert
    h_r\in\Z^n\text{, for  }r=1,\ldots,h\right\}.
  \]
  An application of Lemma \ref{gt:lem3.2} yields
  \begin{gather}\label{mani:2}
  t(\nu)=\sum_{M\in\mathcal{M}}T_Me\left(\sum_{r=1}^h\mathbf{h}_rB^{-l_r-1}\nu\right),
  \end{gather}
  where $T_M=\prod_{r=1}^hc_{m_{r1},\ldots,m_{rn}}$. Combining this with the
  definition of $F_l$ in (\ref{mani:Fl}) we get
  \begin{gather}\label{mani:3}
    \lvert\Theta-\sum_{z\in\mathcal{R}(\mathbf{T})}t(\phi(P(z)))\rvert \leq
    F_{l_1}+\cdots+F_{l_h}.
  \end{gather}

  Plugging (\ref{mani:2}) into (\ref{mani:3}) together with an application of
  Lemma \ref{gt:lem3.2} for the coefficients yields
  \[
  \Theta=
  \frac{c_1\cdots c_t}{\lvert\det(B)\rvert^h}T^n+
  \sum_{0\neq M\in\mathcal{M}}T_Me\left(\sum_{r=1}^h\langle\mathbf{h}_r,B^{-l_r-1}\phi(P(z))\rangle\right)
  +\mathcal{O}\left(\sum_{r=1}^hF_{l_r}\right).
  \]
  Now an application of Proposition \ref{mani:prop:weyl} to treat the exponential
  sums, of Proposition \ref{mani:prop:border} for the border $F_l$ with
  $v\ll\log\log T$ and the observation that
\[
  \sum_{M\in\mathcal{M}}\lvert T_M\rvert\ll\kappa^{-2h}
  \ll\lvert\det B\rvert^{2hv}\ll(\log T)^{t\sigma_0/2},
\]
where we used the definition of $\kappa$ in (\ref{mani:Delta}), proves
the proposition.
\end{proof}

\section{Proof of Theorem \ref{thm:result}}\label{sec:proof-thm}


For this proof we mainly follow the proof of the Theorem of Bassily and K\'atai
\cite{Bassily_Katai1995:distribution_values_q}. In the same manner we cut of
the head and tail of the expansion and show the theorem for a truncated version
of the additive function. In particular we set $C:=\max(C_1,C_2)$, $A:=[C\log
L]$ and $B:=L-A$, where $L$, $C_1$ and $C_2$ are defined in the statement of
Proposition \ref{mani:centralprop}. Furthermore we define the truncated
function $f'$ to be
\[
  f'(P(z))=\sum_{j=A}^Bf(a_j(P(z))b^j).
\]
By the definition of $A$ and $f(ab^j)\ll1$ with $a\in\mathcal{N}$ we get that
$f'(P(z))=f(P(z))+\mathcal{O}(\log L)$. In the same manner we define the
truncated mean and standard deviation
\[
  M'(T):=\sum_{j=A}^Bm_j
  \quad\text{and}\quad
  D'^2(T):=\sum_{j=A}^B\sigma_j^2.
\]

At this point we need that the deviation $D$ tends sufficiently fast do
infinity. In particular, we could refine the statement, if we shrink the part,
which is cut of. Since $M(T)-M'(T)=\mathcal{O}(\log L)$ and
$D^2(T)-D'^2(T)=\mathcal{O}(\log L)$ we get that it suffices to
show that
\[
  \frac1{\#{\mathcal R}(\mathbf{T})}\#\left\{z\in {\mathcal R}(\mathbf{T})\middle\vert
  \frac{f'(P(z))-M'(T^d)}{D'(T^d)}<y\right\}
  \longrightarrow\Phi(y).
\]

By the Fr{\'e}chet-Shohat Theorem (\textit{cf.} \cite[Lemma 1.43]{MR551361})
this holds true if and only if the moments
\[
  \xi_k(T):=\frac1{\#{\mathcal R}(\mathbf{T})}\sum_{z\in {\mathcal R}(\mathbf{T})}\left(
  \frac{f'(P(z))-M'(T^d)}{D'(T^d)}\right)^k
\]
converge to the moments of the normal law for $T\to\infty$. We will show the
last statement by comparing the moments $\xi_k$ with
\[
  \eta_k(T):=\frac1{\#{\mathcal R}(\mathbf{T}^d)}\sum_{z\in N(T^d)}\left(
  \frac{f'(z)-M'(T^d)}{D'(T^d)}\right)^k,
\]
where $\mathbf{T}^d=(T_1^d,\ldots,T_n^d)=(T_{111}^d,\ldots,T_{t,n_t,m_t}^d)$.

An application of Proposition \ref{mani:centralprop} gives that
\[
  \xi_k(T)-\eta_k(T)\to0\quad\text{for}\quad T\to\infty.
\]

Furthermore we get by Proposition
\ref{prop:expansion:length:estimation} that these sums consist of
independently identically distributed random variables (with possible $2C$
exceptions). By the central limit theorem we get that their distribution
converges to the normal law. Thus the $\eta_k(T)$ converge to the moments of
the normal law. This yields
\[
  \lim_{T\to\infty}\xi_k(T)=
  \lim_{T\to\infty}\eta_k(T)=
  \int x^k\mathrm{d}\Phi.
\]

We apply the Fr{\'e}chet-Shohat Theorem again to prove the theorem.

\section*{Acknowledgment}
This paper was written while M. Madritsch was a visitor at the Faculty of
Informatics of the University of Debrecen. He thanks the
centre for its hospitality. During his stay he was supported by the project
HU 04/2010 founded by the \"OAD. The second author was supported by the
Hungarian National Foundation for Scientific Research Grant No.T67580 and by
the T\'AMOP 4.2.1/B-09/1/KONV-2010-0007 project. The second project is
implemented through the New Hungary Development Plan co-financed by the
European Social Fund, and the European Regional Development Fund.

\bibliographystyle{mine}
\def\cprime{$'$}
\providecommand{\bysame}{\leavevmode\hbox to3em{\hrulefill}\thinspace}
\providecommand{\MR}{\relax\ifhmode\unskip\space\fi MR }
\providecommand{\MRhref}[2]{%
  \href{http://www.ams.org/mathscinet-getitem?mr=#1}{#2}
}
\providecommand{\href}[2]{#2}

\end{document}